\documentclass{amsart}

\usepackage{amsmath,amssymb,amsfonts}
\usepackage[all]{xy}
\usepackage{fancyvrb}
\usepackage{appendix,listings,hyperref}

\DeclareMathOperator{\Sym}{Sym}
\DeclareMathOperator{\Hilb}{Hilb}

\newcommand{\hilb}[1]{^{[#1]}}
\newcommand{\ie}{{\it i.e. }}

\newcommand{\vac}{|0\rangle}

\newcommand{\One}{1}

\newcommand{\coloneqq}{:=}

\newcommand{\IC}{\mathbb{C}}
\newcommand{\IR}{\mathbb{R}}
\newcommand{\IQ}{\mathbb{Q}}
\newcommand{\IZ}{\mathbb{Z}}

\newcommand{\km}{\mathfrak{m}}
\newcommand{\kq}{\mathfrak{q}}

\theoremstyle{plain}
\newtheorem{theorem}{Theorem}[section]
\newtheorem{lemma}[theorem]{Lemma}
\newtheorem{proposition}[theorem]{Proposition}
\newtheorem{corollary}[theorem]{Corollary}
\theoremstyle{definition}
\newtheorem{definition}[theorem]{Definition}
\newtheorem{notation}[theorem]{Notation}
\theoremstyle{remark}
\newtheorem{remark}[theorem]{Remark}
\newtheorem{example}[theorem]{Example}

\begin{document}

\title[Products in $H^\ast(\Hilb^n(K3), \IZ)$]{Computing Cup-Products in integral cohomology of Hilbert schemes of points on K3 surfaces}

\author{Simon Kapfer}
\address{Simon Kapfer, Laboratoire de Math\'ematiques et Applications, UMR CNRS 6086, Universit\'e de Poitiers, T\'el\'eport 2, Boulevard Marie et Pierre Curie, F-86962 Futuroscope Chasseneuil}
\email{simon.kapfer@math.univ-poitiers.fr}

\date{\today}


\begin{abstract} 

We study cup 
products in the integral cohomology of the Hilbert scheme of $n$ points on a K3 surface and present a computer program for this purpose. In particular, we deal with the question, which classes can be represented by products of lower degrees.
\end{abstract}

\maketitle


The Hilbert schemes of $n$ points on a complex surface parametrize all zero-dimensional subschemes of length $n$. Studying their rational cohomology, Nakajima \cite{Nakajima} was able to give an explicit description of the vector space structure in terms of the action of a Heisenberg algebra.
The Hilbert schemes of points on a K3 surface are one of the few known classes of Irreducible Holomorphic Symplectic Manifolds. Lehn and Sorger \cite{LehnSorger} developed an algebraic model to describe the cohomological ring structure. On the other hand, Qin and Wang \cite{QinWang} found a base for integral cohomology in the projective case. By combining these results, we are able to compute all cup-products in the cohomology rings of Hilbert schemes of $n$ points on a projective K3 surface with integral coefficients. 
For $n=2$, this was done by Boissi\`ere, Nieper-Wi{\ss}kirchen and Sarti \cite{BNS}, who applied their results to automorphism groups of prime order. When $n$ is increasing, the ranks of the cohomology rings become very large, so we need the help of a computer. The source code is available under \url{https://github.com/s--kapfer/HilbK3}

Our goal here is to give some properties for low degrees.
Denote by $S\hilb{3}$ the Hilbert scheme of 3 points on a projective K3 surface (or a deformation equivalent space). We identify $\Sym^kH^2(S\hilb{n},\IZ)$ with its image in $H^{2k}(S\hilb{n},\IZ)$ under the cup product mapping. 
\begin{theorem}
The cup product mappings for the Hilbert scheme of 3 points on a projective K3 surface have the following cokernels:
\begin{gather}
\frac{H^4(S\hilb{3},\IZ)}{\Sym^2 H^2(S\hilb{3},\IZ)}  \cong \frac{\IZ}{3\IZ} \oplus \IZ^ {\oplus 23}\\
\frac{H^6(S\hilb{3},\IZ)}{H^2(S\hilb{3},\IZ)\smile H^4(S\hilb{3},\IZ)} \cong \left(\frac{\IZ}{3\IZ}\right)^{\oplus 23}
\end{gather}
\end{theorem}
Although the case $n=3$ is the most interesting for us, our computer program allows computations for arbitrary $n$. We give some numerical results in Section \ref{CompSection}.\vspace{5pt}

\emph{Acknowledgements.} The author thanks Samuel~Boissi\`ere and Marc~Nieper-Wi{\ss}kirchen for their supervision and many helpful comments. He also thanks Gr\'{e}goire~Menet for stimulating proposals and the Laboratoire de Math\'{e}matiques et Applications of the university of Poitiers for its hospitality. This work was partially supported by a DAAD grant.
\section{Preliminaries}
\begin{definition}
Let $n$ be a natural number. A partition of $n$ is a decreasing sequence $\lambda = (\lambda_1,\ldots,\lambda_k),\ \lambda_1\geq\ldots\geq\lambda_k>0$ of natural numbers such that $\sum_i \lambda_i =n$. Sometimes it is convenient to write $\lambda = (\ldots,2^{m_2},1^{m_1})$ with multiplicities in the exponent. No confusion should be possible since numerical exponentiation is never meant in this context. We define the weight $\|\lambda\| :=\sum m_i i =n$ and the length $|\lambda| := \sum_i m_i =k$. We also define $z_\lambda \coloneqq\prod_i i^{m_i} m_i!$. 
\end{definition}
\begin{definition} \label{SymFun}
Let $\Lambda_n := \IQ[x_1,\ldots,x_n]^{S_n}$ be the graded ring of symmetric polynomials. There are canonical projections $: \Lambda_{n+1}\rightarrow\Lambda_n$ which send $x_{n+1}$ to zero. The graded projective limit
$\Lambda:=\lim\limits_{\leftarrow}\Lambda_n$ is called the ring of symmetric functions.
Let $m_\lambda$ and $p_\lambda$ denote the monomial and the power sum symmetric functions. They are defined as follows: For a monomial $x_{i_1}^{\lambda_1}x_{i_2}^{\lambda_2}\ldots x_{i_k}^{\lambda_k}$ of total degree $n$, the (ordered) sequence of exponents $(\lambda_1,\ldots,\lambda_k)$ defines a partition $\lambda$ of $n$, which is called the shape of the monomial. Then we define $m_\lambda$ being the sum of all monomials of shape $\lambda$. For the power sums, first define $p_n := x_1^n + x_2^n + \ldots$. Then $p_\lambda := p_{\lambda_1}p_{\lambda_2}\ldots p_{\lambda_k}$.

The families $(m_\lambda)_\lambda$ and $(p_\lambda)_\lambda$ form two $\IQ$-bases of $\Lambda$, so they are linearly related by $p_\lambda = \sum_{\mu} \psi_{\lambda\mu}m_\mu$. It turns out that the base change matrix $(\psi_{\lambda\mu})$ has integral entries, but its inverse $(\psi_{\mu\lambda}^{-1})$ has not. A method to determine the $(\psi_{\lambda\mu})$ is given by Lascoux in \cite[Sect. 3.7]{Lascoux}.
\end{definition}
\begin{definition}
A lattice $L$ is a free $\IZ$-module of finite rank, equipped with a non-degenerate symmetric integral bilinear form $B$. The lattice $L$ is called odd, if there exists a $v\in L$, such that $B(v,v)$ is odd, otherwise it is called even. 
If the map $v \mapsto B(v,v)$ takes both negative and positive values on $L$, the lattice is called indefinite. 
Choosing a base $\{e_i\}_i$ of our lattice, we can write $B$ as a symmetric matrix. $L$ is called unimodular, if the matrix $B$ has determinant $\pm 1$. 
The difference between the number of positive eigenvalues and the number of negative eigenvalues of $B$ (regarded as a matrix over $\IR$) is called the signature.
\end{definition}
There is the following classification theorem. See \cite[Chap. II]{milnor1973symmetric} for reference.
\begin{theorem}
Two indefinite unimodular lattices $L$, $L'$ are isometric iff they have the same rank, signature and parity. Evenness implies that the signature is divisible by 8. 
In particular, if $L$ is odd, then $L$ possesses an orthogonal basis and is hence isometric to $\left<1\right>^{\oplus k}\oplus\left<-1\right>^{\oplus l}$ for some $k,l\geq 0$. If $L$ is even, then $L$ is isometric to $U^{\oplus k}\oplus (\pm E_8)^{\oplus l}$ for some $k,l\geq 0$.
\end{theorem}
\begin{definition}
Let $S$ be a projective K3 surface. We fix integral bases $\One$ of $H^0(S,\IZ)$, $x$ of $H^4(S,\IZ)$ and $\alpha_1,\ldots ,\alpha_{22}$ of $H^2(S,\IZ)$. The cup product induces a symmetric bilinear form $B_{H^2}$ on $H^2(S,\IZ)$ and thus the structure of a unimodular lattice.
We may extend $B_{H^2}$ to a symmetric non-degenerate bilinear form $B$ on $H^\ast(S,\IZ)$ by setting $ B(\One,\One) = 0,\ B(\One,\alpha_i) = 0,\ B(\One,x) = 1, \ B(x,x) = 0$.
\end{definition}
By the Hirzebruch index theorem, we know that $H^2(S,\IZ)$ has signature $-16$ and, by the classification theorem for indefinite unimodular lattices, is isomorphic to $U^{\oplus 3}\oplus (-E_8)^{\oplus 2}$.
\begin{definition}\label{comult}
$B$ induces a form $B\otimes B$ on $\Sym^2H^\ast(S,\IZ)$. So the cup-product 
\begin{equation*}
\mu : \Sym ^2H^{*}(S,\IZ) \longrightarrow H^\ast(S,\IZ) 
\end{equation*}
induces an adjoint comultiplication $\Delta$ that is coassociative, given by:
\begin{equation*}
\Delta : H^\ast(S,\IZ) \longrightarrow \Sym^2H^\ast(S,\IZ),\quad \Delta = -(B\otimes B)^{-1}\mu^TB
\end{equation*}
with the property $(B\otimes B)\left(\Delta(a),b\otimes c\right)=-B\left(a,b\smile c\right)$. Note that this does not define a bialgebra structure.
The image of $\One$ under the composite map $\mu\circ\Delta$, denoted by $e=24x$ is called the Euler Class.

More generally, every linear map $f: A^{\otimes k} \rightarrow A^{\otimes m}$ induces an adjoint map $g$ in the other direction that satisfies $(-1)^mB^{\otimes m}(f(x),y)= (-1)^{k} B^{\otimes k}(x,g(y))$.
\end{definition}
We denote by $S\hilb{n}$ the Hilbert scheme of $n$ points on $S$, \ie the classifying space of all zero-dimensional closed subschemes of length $n$. $S\hilb{0}$ consists of a single point and $S\hilb{1}=S$. Fogarty \cite[Thm.~2.4]{Fogarty} proved that the Hilbert scheme is a smooth variety.
A theorem by Nakajima \cite{Nakajima} gives an explicit description of the vector space structure of $H^\ast(S\hilb{n},\IQ)$ in terms of creation operators
$$
\kq_l(\beta) :  H^\ast(S\hilb{n},\IQ) \longrightarrow  H^{\ast+k+2(l-1)}(S\hilb{n+l},\IQ)
,$$ 
where $\beta\in H^k(S,\IQ)$, acting on the direct sum 
$\mathbb{H}:=\bigoplus_n H^\ast(S\hilb{n},\IQ)$. The operators $\kq_l(\beta)$ are linear and commute with each other. The vacuum vector $\vac$ is defined as the generator of $H^0(S\hilb{0},\IQ)\cong\IQ$. The images of $\vac$ under the polynomial algebra generated by the creation operators span $\mathbb{H}$ as a vector space. 
Following \cite{QinWang}, we abbreviate $\kq_{l_1}(\beta)\ldots\kq_{l_k}(\beta)=:\kq_\lambda(\beta)$, where the partition $\lambda$ is composed by the $l_i$. 

An integral basis for $H^\ast(S\hilb{n},\IZ)$ in terms of Nakajima's operators was given by Qin--Wang:
\begin{theorem} \label{QinWangTheorem}\cite[Thm. 5.4.]{QinWang} Let $\km_{\nu,\alpha} := \sum_\rho \psi_{\nu\rho}^{-1}\,\kq_{\rho}(\alpha)$, with coefficients $ \psi_{\nu\rho}^{-1}$ as in Definition \ref{SymFun}. The classes
$$ \frac{1}{z_\lambda} \kq_\lambda(1)\kq_\mu(x)\km_{\nu^1,\alpha_1}\ldots\km_{\nu^{22},\alpha_{22}}\vac,\quad \|\lambda\| +\|\mu\| + \sum_{i=1}^{22}\|\nu^i\| = n
$$ 
form an integral basis for $H^\ast(S\hilb{n},\IZ)$. Here,
$\lambda,\; \mu,\; \nu^i$ are partitions.
\end{theorem}
\begin{notation}\label{notation}
To enumerate the basis of $H^\ast(S\hilb{n},\IZ)$, we introduce the following abbreviation:
$$ \boldsymbol\alpha^{\boldsymbol\lambda} :=
\One^{\lambda^0} \alpha_1^{\lambda^1}\ldots\alpha_{22}^{\lambda^{22}}x^{\lambda^{23}} :=
\frac{1}{z_{\widetilde{\lambda^0}} }
\kq_{\widetilde{\lambda^0}}(\One)\kq_{\lambda^{23}}(x)\km_{\lambda^1,\alpha_1}\ldots\km_{\lambda^{22},\alpha_{22}}\vac
$$
where the partition $\widetilde{\lambda^0}$ is built from $\lambda^0$ by appending sufficiently many ones, such that $\left\|\widetilde{\lambda^0}\right\| +\sum_{i\geq 1}\left\|\lambda^i\right\| = n $. If $\sum_{i\geq 0}\left\|\lambda^i\right\| > n, $ we put $\boldsymbol\alpha^{\boldsymbol\lambda}=0$. Thus we can interpret $\boldsymbol\alpha^{\boldsymbol\lambda}$ as an element of $H^\ast(S\hilb{n},\IZ)$ for arbitrary $n$. We say that the symbol $\boldsymbol\alpha^{\boldsymbol\lambda}$ is reduced, if $\lambda^0$ contains no ones. We define also $\left\|\boldsymbol\lambda\right\| := \sum_{i\geq 0}\left\|\lambda^i\right\|$. 
\end{notation}
\begin{lemma}\label{degBound}
Let $\boldsymbol\alpha^{\boldsymbol\lambda}$ represent a class of cohomological degree $2k$. If $\boldsymbol\alpha^{\boldsymbol\lambda}$ is reduced, then $\frac{k}{2}\leq\left\|\boldsymbol{\lambda}\right\| \leq 2k$.
\begin{proof} This is a simple combinatorial observation. We give the two extremal cases.
The lowest ratio between $\left\|\boldsymbol{\lambda}\right\|$ and $\deg \boldsymbol\alpha^{\boldsymbol\lambda}$ is achieved by the classes $x^{(1^m)}$, where the degree is $4m$ and the weight of $\boldsymbol{\lambda}$ is $m$. The highest ratio is achieved by the classes $1^{(2^m)}$, where both degree and weight equal $2m$. So $\frac{1}{4}\leq\frac{\left\|\boldsymbol{\lambda}\right\|}{\deg \boldsymbol\alpha^{\boldsymbol\lambda}}\leq 1$.
\end{proof}
\end{lemma}
The ring structure of $H^\ast(S\hilb{n}, \IQ)$ has been studied by Lehn and Sorger in \cite{LehnSorger}, where an explicit algebraic model is constructed, which we recall briefly:
\begin{definition} \label{model}\cite[Sect. 2]{LehnSorger}
Let $\pi$ be a permutation of $n$ letters, written as a product of disjoint cycles. To each cycle we may associate an element of $A:=H^\ast(S,\IQ)$. This defines an element in $A^{\otimes m}$, $m$ being the number of cycles. For example, a term like $(1\,2\,3)_{\alpha_1}(4\,5)_{\alpha_2}$ may describe a permutation consisting of two cycles with associated classes $\alpha_1,\alpha_2\in A$. We interpret the cycles as the orbits of the subgroup $\left<\pi\right>\subset S_n$ generated by $\pi$. We denote the set of orbits by $\left<\pi\right>\backslash[n]$. Thus we construct a vector space $A\{S_n\}:=\bigoplus_{\pi\in S_n} A^{\otimes\left<\pi\right>\backslash[n]}$. 

To define a ring structure, take two permutations $\pi,\,\tau \in S_n$ and the subgroup $\left< \pi,\tau\right>$ generated by them. The natural map of orbit spaces
$
p_\pi:\left<\pi\right>\backslash[n] \rightarrow \left<\pi,\tau\right>\backslash[n]
$
induces a map $f^{\pi,\left<\pi,\tau\right>} : A^{\otimes\left<\pi\right>\backslash[n]} \rightarrow A^{\otimes\left<\pi,\tau\right>\backslash[n]}$, which multiplies the factors of an elementary tensor if the corresponding orbits are glued together.
Denote $f_{\left<\pi,\tau\right>,\pi} $ the adjoint to this map in the sense of Definition \ref{comult}. Then the map
\begin{gather*}
m_{\pi,\tau} : A^{\otimes\left<\pi\right>\backslash[n]} \otimes A^{\otimes\left<\tau\right>\backslash[n]} \longrightarrow A^{\otimes\left<\pi\tau\right>\backslash[n]} ,  \\
a\otimes b \longmapsto  f_{\left<\pi,\tau\right>,\pi\tau} (f^{\pi,\left<\pi,\tau\right>} (a)\cdot f^{\tau,\left<\pi,\tau\right>}(b)\cdot  e^{g(\pi,\tau)} ) 
\end{gather*}
defines a multiplication on $A\{S_n\}$. Here the dot means the cup product on each tensor factor and $e^{g(\pi,\tau)} \in A^{\otimes\left<\pi,\tau\right>\backslash[n]}$ is an elementary tensor that is composed by powers of the Euler class $e$: for each orbit $B \in  {\otimes\left<\pi,\tau\right>\backslash[n]}$ the exponent $g(\pi,\tau)(B)$ (so-called "graph defect", see \cite[2.6]{LehnSorger}) is given by:
$$
g(\pi,\tau)(B) = \frac{1}{2}\left( |B| +2 - |p_\pi^{-1}(\{B\})|- |p_\tau^{-1}(\{B\})| - |p_{\pi\tau}^{-1}(\{B\})|  \right).
$$

\end{definition}
The symmetric group $S_n$ acts on $A\{S_n\}$ by conjugation, permuting the direct summands: conjugation by $\sigma\in S_n$ maps $A^{\otimes\left<\pi\right>\backslash[n]}$ to $A^{\otimes\left<\sigma\pi\sigma^{{-}1}\right>\backslash[n]}$.
This action preserves the ring structure. Therefore the space of invariants $A\hilb{n} := \left(A\{S_n\} \right)^{S_n}$ becomes a subring. The main theorem of \cite{LehnSorger} can now be stated:
\begin{theorem} \label{LSThm}\cite[Thm. 3.2.]{LehnSorger}
The following map is an isomorphism of rings:
\begin{align*}
H^\ast(S\hilb{n},\IQ) & \longrightarrow A\hilb{n} \\
\kq_{n_1}(\beta_1)\ldots \kq_{n_k}(\beta_k) \vac &\longmapsto \sum_{\sigma\in S_n} \sigma a \sigma^{{-}1} 
\end{align*}
with $\sum_i n_i=n$ and $a =(1\,2\ldots n_1)_{\beta_1}(n_1\! +\! 1\ldots n_1\!+\! n_2)_{\beta_2}\cdots (n\!-\!n_k \ldots n)_{\beta_k}\in A\{S_n\}$.
\end{theorem}

Since $H^\text{odd}(S\hilb{n},\IZ) = 0$ and $H^\text{even}(S\hilb{n},\IZ)$ is torsion-free by \cite{Markman}, we can apply these results to $H^\ast(S\hilb{n}, \IZ)$ to determine the multiplicative structure of cohomology with integer coefficients. It turns out, that it is somehow independent of $n$. More precisely, we have the following stability theorem, by Li, Qin and Wang:
\begin{theorem} \label{stability}\emph{(Derived from \cite[Thm.~2.1]{QinWang}).}
Let $Q_1,\ldots,Q_s$ be products of creation operators, \ie $Q_i = \prod_j \kq_{\lambda_{i,j}}(\beta_{i,j}) $ for some partitions $\lambda_{i,j}$ and classes $\beta_{i,j}\in H^\ast(S,\IZ)$. Set $n_i := \sum_j \left\|\lambda_{i,j}\right\|$.
Then the cup product 
$ \prod_{i=1}^s \left(\frac{1}{(n-n_i)!} \kq_{1^{n-n_i}}(1) \,Q_i \,\vac \right)$ equals a finite linear combination of classes of the form $\frac{1}{(n-m)!}\kq_{1^{n-m}}(1)\prod_j \kq_{\mu_{j}}(\gamma_{j})\vac$, with $\gamma\in H^\ast(S,\IZ)$, $m=\sum_j\left\|\mu_j\right\|$, whose coefficients are independent of $n$. We have the upper bound $m\leq\sum_i n_i$. Moreover, $m=\sum_i n_i$ if and only if the corresponding class is $\frac{1}{(n-m)!}\kq_{1^{n-m}}(1)Q_1Q_2\ldots Q_s\vac$ with coefficient $1$.
\end{theorem}
\begin{corollary} \label{stabCor} Let $\boldsymbol{\alpha}^{\boldsymbol{\lambda}},\boldsymbol{\alpha}^{\boldsymbol{\mu}},\boldsymbol{\alpha}^{\boldsymbol{\nu}}$ be reduced. Assume $n\geq\left\|\boldsymbol\lambda\right\|,\left\|\boldsymbol\mu\right\| $. Then the coefficients $c^{\boldsymbol{\lambda\mu}}_{\boldsymbol{\nu}}$ of the cup product in $H^\ast(S\hilb{n},\IZ)$
$$\boldsymbol{\alpha}^{\boldsymbol{\lambda}} \smile
\boldsymbol{\alpha}^{\boldsymbol{\mu}}
= \sum_{\boldsymbol{\nu}} c^{\boldsymbol{\lambda\mu}}_{\boldsymbol{\nu}} \boldsymbol{\alpha}^{\boldsymbol{\nu}}
$$  
are polynomials in $n$ of degree at most $ \left\|\boldsymbol\lambda\right\|+\left\|\boldsymbol\mu\right\|-\left\|\boldsymbol\nu\right\|$.
\end{corollary}
\begin{proof} Set $Q_{\boldsymbol{\lambda}}:=  \kq_{\lambda^0}(1)\kq_{\lambda^{23}}(x)\prod_{1\leq j\leq 22}\kq_{\lambda^j}(\alpha_j)$ and $n_{\boldsymbol\lambda}:=\left\|\boldsymbol\lambda\right\|$. Then we have:
$\boldsymbol{\alpha}^{\boldsymbol{\lambda}} = \frac{1}{(n-n_{\boldsymbol\lambda})!\,z_{\lambda^0}}\kq_{1^{n-n_{\boldsymbol\lambda}}}(1)Q_{\boldsymbol\lambda}\vac$ and $
\boldsymbol{\alpha}^{\boldsymbol{\mu}}=\frac{1}{(n-n_{\boldsymbol\mu})!\,z_{\mu^0}}\kq_{1^{n-n_{\boldsymbol\mu}}}(1)Q_{\boldsymbol\mu}\vac $. 
Thus the coefficient $c^{\boldsymbol{\lambda\mu}}_{\boldsymbol{\nu}}$ in the product expansion is a constant, which depends on $ \left\|\boldsymbol\lambda\right\|$, $\left\|\boldsymbol\mu\right\|$, $\left\|\boldsymbol\nu\right\|$, but not on $n$, multiplied with $\frac{(n-n_{\boldsymbol\nu})!}{(n-m)!}$ for a certain $m\leq n_{\boldsymbol\lambda}+n_{\boldsymbol\mu}$. 
This is a polynomial of degree $m-n_{\boldsymbol\nu}\leq n_{\boldsymbol\lambda}+n_{\boldsymbol\mu}-n_{\boldsymbol\nu} =\left\|\boldsymbol\lambda\right\|+\left\|\boldsymbol\mu\right\|-\left\|\boldsymbol\nu\right\| $.
\end{proof}
\begin{remark}
If $n<\left\|\boldsymbol\lambda\right\|$ or $n<\left\|\boldsymbol\mu\right\| $, one has $\boldsymbol{\alpha}^{\boldsymbol{\lambda}}=0$, resp.~$\boldsymbol{\alpha}^{\boldsymbol{\mu}}=0$. But it is still possible that $\boldsymbol{\alpha}^{\boldsymbol{\nu}}\neq 0$ in $H^*(S\hilb{n})$. It seems that in this case the polynomial $ c^{\boldsymbol{\lambda\mu}}_{\boldsymbol{\nu}}$ always becomes zero when evaluated at $n$. So the $ c^{\boldsymbol{\lambda\mu}}_{\boldsymbol{\nu}}$ seem to be universal in the sense that the above corollary holds true even without the condition $n\geq\left\|\boldsymbol\lambda\right\|,\left\|\boldsymbol\mu\right\| $.
\end{remark}
\begin{example}\label{example} Here are some explicit examples for illustration. See \ref{exampleSource} for how to compute them. 
\begin{enumerate} \item $
\One^{(2,2)}\smile \alpha_i^{(2)} = -2\cdot \One^{(2)}\alpha_i^{(1)}x^{(1)} + \One^{(2,2)}\alpha_i^{(2)} + 2\cdot\One^{(2)}\alpha_i^{(3)} +\alpha_i^{(4)} $ for $i\in\{1..22\}$.
\item Let $i,j\in\{1\ldots 22\}$. 
If $i \neq j$, then $\alpha_i^{(2)}\smile\alpha_j^{(1)} = \alpha_i^{(2)}\alpha_j^{(1)} + 2B(\alpha_i,\alpha_j)\cdot x^{(1)}$. 
Otherwise, $\alpha_i^{(2)}\smile\alpha_i^{(1)} = \alpha_i^{(3)}+ \alpha_i^{(2,1)} + 2B(\alpha_i,\alpha_i)\cdot x^{(1)}$.
\item Set $\boldsymbol{\alpha}^{\boldsymbol{\lambda}} = \One^{(2)}$ and $\boldsymbol{\alpha}^{\boldsymbol{\nu}}=x^{(1)}$. Then $c^{\boldsymbol{\lambda\lambda}}_{\boldsymbol{\nu}} = -(n-1)$.
\item Set $\boldsymbol{\alpha}^{\boldsymbol{\lambda}} = \One^{(2,2)}$ and $\boldsymbol{\alpha}^{\boldsymbol{\nu}}=x^{(1,1)}$. Then $c^{\boldsymbol{\lambda\lambda}}_{\boldsymbol{\nu}} =\frac{(n-3)(n-2)}{2}$.
\end{enumerate}
\end{example}
\begin{example} \label{oddWitness} Let $i,j$ be indices, such that $B(\alpha_i,\alpha_j)=1,\ B(\alpha_i,\alpha_i)=0=B(\alpha_j,\alpha_j)$ and let $k\geq 0$. Set $\boldsymbol{\alpha}^{\boldsymbol{\lambda}} = \alpha_i^{(1)}\alpha_j^{(1)}x^{(1^k)}$ 
and $\boldsymbol{\alpha}^{\boldsymbol{\nu}}= x^{(1^{2k+2})}$. Then $c^{\boldsymbol{\lambda\lambda}}_{\boldsymbol{\nu}} =1$.
\end{example}
\begin{proof}
It is not hard to see from the definition, that for $\beta_j,\;\gamma_j\in H^*(S)$:
$$
\kq_1(\beta_1)\ldots\kq_1(\beta_n)\vac \smile\kq_1(\gamma_1)\ldots\kq_1(\gamma_n)\vac = \sum_{\sigma \in S_n} \kq_1(\beta_1\cdot\gamma_{\sigma(1)})\ldots\kq_1(\beta_n\cdot\gamma_{\sigma(n)})\vac.
$$ 
A combinatorial investigation yields now:
$$
\left(\kq_1(\alpha_i)\kq_1(\alpha_j)\kq_1(x)^k\kq_1(1)^{k+m}\vac \right)^2 = \frac{(k+m)!^2}{m!} \kq_1(x)^{2k+2}\kq_1(1)^m\vac + \text{other terms}.
$$
Looking at \ref{notation}, the result follows.
\end{proof}
\begin{theorem}\label{freeness}
The quotient
$$
 \frac{H^{2k}(S\hilb{n},\IZ)}{\Sym^k H^{2}(S\hilb{n},\IZ)}
$$
is a free $\IZ$-module for $n\geq k+2$.
\end{theorem}
\begin{proof}
The idea of the proof is to modify the basis of $H^{2k}(S\hilb{n},\IZ)$, given in Theorem~\ref{QinWangTheorem}, in a way that $\Sym^k H^{2}(S\hilb{n},\IZ)$ splits as a direct summand. 

Given a free $\IZ$-module $M$ with basis $(b_i)_{i=1\ldots m}$ and a vector $v = a_1b_1 + \ldots + a_mb_m$. Then there is another basis of $M$ which contains $v$, iff $\gcd\{a_1,\ldots,a_m\} = 1$. More generally, given a set of vectors $(v_i)_{i=1\ldots r}$, $v_i=a_{i1}b_1+\ldots+a_{im}b_m$, we can complete it to a basis of $M$, iff the $r\times r$-minors of the matrix $(a_{ij})_{ij}$ share no common divisor. We want to show that the canonical basis of $\Sym^k H^{2}(S\hilb{n},\IZ)$ is such a set.

A basis of $ H^{2}(S\hilb{n},\IZ)$ is given by the classes $\alpha_i^{(1)}=\frac{1}{(n-1)!}\kq_{1^{n-1}}(1)\kq_1(\alpha_i)\vac$, $i=1,\ldots ,22$ and $1^{(2)} =\frac{1}{2(n-2)!}\kq_{(2,1^{n-2})}(1)\vac$.
A power of $\alpha_i^{(1)}$ looks like (Thm.~\ref{stability}):
\begin{align*}
\left(\alpha_i^{(1)}\right)^k & =\frac{1}{(n-k)!}\kq_{1^{n-k}}(1)\kq_{1^k}(\alpha_i)\vac + \text{other terms containing } \kq_\lambda(x).
\end{align*}
Now, by the definition of $\psi_{\lambda\mu}$, $\kq_{1^k}(\alpha_i) = \km_{(k),\alpha_i} + \ldots + k! \cdot\km_{(1^k),\alpha_i}$, so
\begin{equation}
\left(\alpha_i^{(1)}\right)^k  = \alpha_i^{(k)} + \text{other terms}.
\end{equation}
Next, we determine the coefficients of $1^{(k+1)}$ and $1^{(k,2)}$ in the expansion of $\left(1^{(2)}\right)^k$. Considering Definition \ref{model}, we observe that here the graph defect is zero and the adjoint map is trivial, so the problem reduces to combinatorics of the symmetric group: the coefficient of $1^{(k+1)}$ is the number of ways to write a $(k+1)$-cycle as a product of $k$ transpositions. A result of D\'enes \cite{Denes} states that this is $(k\!+\!1)^{k-1}$. For the $1^{(k,2)}$-coefficient, we have to choose one transposition, and write a $k$-cyle as a product of the remaining $k-1$ transpositions. The number of possibilities is therefore $k\cdot k^{k-2} = k^{k-1}$. So
\begin{equation}
\left(1^{(2)}\right)^k = (k\!+\! 1)^{k-1} \cdot 1^{(k+1)} \;+\; k^{k-1}\cdot 1^{(k,2)} \;+\; \text{other terms}.
\end{equation}
Note that these two coefficients are coprime. 
Putting the two cases together, one gets for a general element of $\Sym^kH^2(S\hilb{n},\IZ)$, $k=k_0+\ldots+k_{22}$:
\begin{align*}
\left(1^{(2)}\right)^{k_0}\prod_{i=1}^{22}\left(\alpha_i^{(1)}\right)^{k_i} &=  (k_0\!+\! 1)^{k_0-1} \cdot 1^{(k_0+1)}\alpha_1^{(k_1)}\ldots \alpha_{22}^{(k_{22})} \\
&+k_0^{k_0-1}\cdot 1^{(k_0,2)}\alpha_1^{(k_1)}\ldots \alpha_{22}^{(k_{22})} +\text{other terms}.
\end{align*}
One checks, that this is the only element of $\Sym^kH^2(S\hilb{n},\IZ)$ having a nonzero coefficient at $1^{(k_0+1)}\alpha_1^{(k_1)}\ldots \alpha_{22}^{(k_{22})}$ and $1^{(k_0,2)}\alpha_1^{(k_1)}\ldots \alpha_{22}^{(k_{22})}$. Now it is easy to show the existence of a complementary basis.
\end{proof}

\section{Computational results} \label{CompSection}
We now give some results in low degrees, obtained by computing multiplication matrices with respect to the integral basis of $H^*(S\hilb{n},\IZ)$. To get their cokernels, one has to reduce them to Smith normal form. Both results have been obtained using a computer.
\begin{remark}
Denote $h^k(S\hilb{n})$ the rank of $H^k(S\hilb{n},\IZ)$. We have:
\begin{itemize}
\item $h^2(S\hilb{n}) = 23 $ for $n\geq 2$.
\item $h^4(S\hilb{n}) = 276,\; 299,\; 300$ for $n=2,3, \geq 4$ resp.
\item $h^6(S\hilb{n}) = 23,\; 2554,\; 2852,\; 2875,\; 2876$ for $n=2,3,4,5,\geq6$ resp.
\end{itemize}
\end{remark}
The algebra generated by classes of degree 2 is an interesting object to study. For cohomology with complex coefficients, Verbitsky has proven in \cite{Verbitsky} that the cup product mapping from $\Sym^k H^2(S\hilb{n},\IC)$ to $H^{2k}(S\hilb{n},\IC)$ is injective for $k\leq n$. Since there is no torsion, one concludes that this also holds for integral coefficients.
\begin{proposition} We identify $\Sym^2H^2(S\hilb{n},\IZ)$ with its image in $H^4(S\hilb{n},\IZ)$ under the cup product mapping. Then: 
\setcounter{equation}{0} 
\begin{align}
\frac{H^4(S\hilb{2},\IZ)}{\Sym^2 H^2(S\hilb{2},\IZ)} & \cong \left(\frac{\IZ}{2\IZ}\right)^{\oplus 23} \oplus \frac{\IZ}{5\IZ},\\
\label{sym23}
\frac{H^4(S\hilb{3},\IZ)}{\Sym^2 H^2(S\hilb{3},\IZ)} & \cong \frac{\IZ}{3\IZ} \oplus \IZ^ {\oplus 23}, \\
\frac{H^4(S\hilb{n},\IZ)}{\Sym^2 H^2(S\hilb{n},\IZ)} & \cong  \IZ^ {\oplus 24}, \quad \text{for }n\geq 4.
\end{align}
The 3-torsion part in (\ref{sym23}) is generated by the integral class $1^{(3)}$.
\end{proposition}
\begin{remark}
The torsion in the case $n=2$ was also computed by Boissi\`{e}re, Nieper-Wi\ss kirchen and Sarti, \cite[Prop. 3]{BNS} using similar techniques.
For all the author knows, the result for $n=3$ is new.
The freeness result for $n\geq 4$ was already proven by Markman, \cite[Thm. 1.10]{Markman2}, using a completely different method. 
\end{remark}
\begin{proposition} For triple products of $H^2(S\hilb{n},\IZ)$, we have:
$$
\frac{H^6(S\hilb{2},\IZ)}{\Sym^3 H^2(S\hilb{2},\IZ)} \cong 
\frac{\IZ}{2\IZ}.
$$
The quotient is generated by the integral class $x^{(2)}$. Moreover,
$$
\frac{H^6(S\hilb{3},\IZ)}{\Sym^3 H^2(S\hilb{3},\IZ)} \cong  \left(\frac{\IZ}{2\IZ}\right)^{\oplus 230}\oplus \left(\frac{\IZ}{36\IZ}\right)^{\oplus 22}\oplus \frac{\IZ}{72\IZ} \oplus \IZ^{\oplus 254},
$$
$$
\frac{H^6(S\hilb{4},\IZ)}{\Sym^3 H^2(S\hilb{4},\IZ)} \cong  \frac{\IZ}{2\IZ} \oplus \IZ^{\oplus 552}.
$$
For $n\geq 5$, the quotient is free by Theorem \ref{freeness}.
\end{proposition}
We study now cup products between classes of degree 2 and 4. The case of $S\hilb{3}$ is of particular interest.
\begin{proposition} \label{p24}The cup product mapping $ : H^2(S\hilb{n},\IZ)\otimes H^4(S\hilb{n},\IZ) \rightarrow H^6(S\hilb{n},\IZ) $ is neither injective (unless $n=0$) nor surjective (unless $n\leq 2$). We have:
\setcounter{equation}{0} 
\begin{align} 
\frac{H^6(S\hilb{3},\IZ)}{H^2(S\hilb{3},\IZ)\smile H^4(S\hilb{3},\IZ)} &\cong \left(\frac{\IZ}{3\IZ}\right)^{\oplus 22} \oplus \frac{\IZ}{3\IZ},
\\
\frac{H^6(S\hilb{4},\IZ)}{H^2(S\hilb{4},\IZ)\smile H^4(S\hilb{4},\IZ)} &\cong  \left(\frac{\IZ}{6\IZ}\right)^{\oplus 22}\oplus\frac{\IZ}{108\IZ} \oplus\frac{\IZ}{2\IZ} ,
\\
\frac{H^6(S\hilb{5},\IZ)}{H^2(S\hilb{5},\IZ)\smile H^4(S\hilb{5},\IZ)} &\cong 
 \IZ^{\oplus 22} \oplus \IZ,
\\
\frac{H^6(S\hilb{n},\IZ)}{H^2(S\hilb{n},\IZ)\smile H^4(S\hilb{n},\IZ)} &\cong 
 \IZ^{\oplus 22} \oplus \IZ\oplus\IZ, \ n\geq 6.
\end{align}
In each case, the first 22 factors of the quotient are generated by the integral classes 
 $$
\alpha_i^{(1,1,1)} -3\cdot \alpha_i^{(2,1)} + 3\cdot \alpha_i^{(3)}+ 3 \cdot \One^{(2)}\alpha_i^{(1,1)} -6\cdot \One^{(2)}\alpha_i^{(2)}+6\cdot \One^{(2,2)}\alpha_i^{(1)}-3\cdot \One^{(3)}\alpha_i^{(1)},
$$ 
for $ i=1\ldots 22$. Now define an integral class
\begin{align*}
K:=&\;\sum_{i\neq j} B(\alpha_i,\alpha_j)\left[\alpha_i^{(1,1)}\alpha_j^{(1)} - 2\cdot\alpha_i^{(2)}\alpha_j^{(1)}+\frac{3}{2}\cdot \One^{(2)}\alpha_i^{(1)}\alpha_j^{(1)} \right] +\\
+&\;\sum_{i}B(\alpha_i,\alpha_i)\left[\alpha_i^{(1,1,1)} - 2\cdot\alpha_i^{(2,1)} + \frac{3}{2}\cdot \One^{(2)}\alpha_i^{(1,1)} \right]+  x^{(2)}-\One^{(2)}x^{(1)}.
\end{align*} 
In the case $n=3$, the last factor of the quotient is generated by $K$. 
\\In the case $n=4$, the class $ \One^{(4)}$ generates the 2-torsion factor and $K-38\cdot\One^{(4)}$ generates the 108-torsion factor.
\\In the case $n=5$, the last factor of the quotient is generated by $K - 16\cdot \One^{(4)} + 21\cdot \One^{(3,2)}$.\\
If $n\geq 6$, the two last factor of the quotient are generated over the rationals by $K +\frac{4}{3}(45-n)\One^{(2,2,2)} - (48-n)\One^{(3,2)}$ and $K+\frac{1}{2}(40-n)\One^{(2,2,2)}- \frac{1}{4}(48-n)\One^{(4)}$. Over $\IZ$, one has to take appropriate multiples depending on $n$, such that the coefficients become integral numbers.
\end{proposition}
\begin{proof} The last assertion for arbitrary $n$ follows from Corollary \ref{stabCor}. First observe that for $\boldsymbol{\alpha}^{\boldsymbol{\lambda}}\!\in\! H^2,\  \boldsymbol{\alpha}^{\boldsymbol{\mu}}\!\in\! H^4,\  \boldsymbol{\alpha}^{\boldsymbol{\nu}}\!\in\! H^6 $, we have $\left\| \boldsymbol\lambda\right\| \leq 2$, $\left\| \boldsymbol\mu\right\| \leq 4$ and $\left\| \boldsymbol\nu\right\| \geq 2,$ according to Lemma \ref{degBound}.
The coefficients of the cup product matrix are thus polynomials of degree at most $2+4-2 =4$ and it suffices to compute only a finite number of instances for $n$. It turns out that the maximal degree is $1$ and the cokernel of the multiplication map is given as stated.
\end{proof}

In what follows, we compare some well-known facts about Hilbert schemes of points on K3 surfaces with our numerical calculations. This means, we have some tests that may justify the correctness of our computer program.
We state now computational results for the middle cohomology group. Since $S\hilb{n}$ is a projective variety of complex dimension $2n$, Poincar\'{e} duality gives $H^{2n}(S\hilb{n},\IZ)$ the structure of a unimodular lattice.  
\begin{proposition} Let $L$ denote the unimodular lattice $H^{2n}(S\hilb{n},\IZ)$. We have:
\begin{enumerate}
\item For $n=2$, $L$ is an odd lattice of rank $276$ and signature $156$.
\item For $n=3$, $L$ is an even lattice of rank $2554$ and signature $-1152$.
\item For $n=4$, $L$ is an odd lattice of rank $19298$ and signature $7082$.
\end{enumerate}
For $n$ even, $L$ is always odd. 
\end{proposition} 
\begin{proof}The numerical results come from an explicit calculation. For $n$ even, we always have the norm-1-vector given by Example \ref{oddWitness}, so $L$ is odd. To obtain the signature, we could equivalently use Hirzebruch's signature theorem and compute the L-genus of $S\hilb{n}$. For the signature, we need nothing but the Pontryagin numbers, which can be derived from the Chern numbers of $S\hilb{n}$. These in turn are known by Ellingsrud, G\"ottsche and Lehn, \cite[Rem. 5.5]{EGL}. 
\end{proof}
Another test is to compute the lattice structure of $H^2(S\hilb{2},\IZ)$, with bilinear form given by $(a,b)\longmapsto \int \left(a\smile b\smile \One^{(2)}\smile \One^{(2)}\right)$. The signature of this lattice is $17$, as shown by Boissi\`ere, Nieper-Wi{\ss}kirchen and Sarti \cite[Lemma 6.9]{BNS}.

\lstset{
  language={Haskell},
  basicstyle=\tiny,
  tabsize=2,
  basewidth=0.53em
}
\appendix
\section{Source Code}
We give the source code for our computer program. It is available online under \url{https://github.com/s--kapfer/HilbK3}. We used the language Haskell, compiled with the \textsc{GHC} software, version 7.6.3. We make use of two external packages: \textsc{permutation} and \textsc{MemoTrie}. The project is divided into 4 modules. 

\subsection{How to use the code}
The main module is in the file \verb|HilbK3.hs|, which can be opened by \textsc{GHCI} for interactive use. It provides an implementation of the ring structure of $A\hilb{n}=H^*(S\hilb{n},\IQ)$, for all $n\in\mathbb{N}$. 
It computes cup--products in reasonable time up to $n=8$.
A product of Nakajima operators is represented by a pair consisting of a partition of length $k$ and a list of the same length, filled with indices for the basis elements of $H^*(S)$. For example, the class
$$
\kq_{3}(\alpha_6)\kq_{3}(\alpha_7)\kq_2(x)\kq_{1}(\alpha_2)\kq_{1}(1)^2 \vac 
$$
in $H^{20}(S\hilb{11})$ is written as
\begin{verbatim}
*HilbK3> (PartLambda [3,3,2,1,1,1], [6,7,23,2,0,0]) :: AnBase 
\end{verbatim}
Note that the classes $1\in H^0(S)$ and $x\in H^4(S)$ have indices \verb|0| and \verb|23| in the code.
The multiplication in $A\hilb{n}$ is implemented by the method \verb|multAn|. 

The classes from Theorem \ref{QinWangTheorem} are represented in the same format, as shown in the following example. The multiplication in $H^*(S\hilb{n},\IZ)$ of such classes is implemented by the method \verb|cupInt|.
\begin{example}\label{exampleSource} We want to compute the results from Example \ref{example}. We only do one particular instance for every example, since the others are similar. By Corollary \ref{stabCor}, it suffices to know the values for finitely many $n$ to deduce the general case.
\begin{enumerate}
\item We do the case $n=6,\ i=1$. 
\begin{Verbatim}[fontsize=\small]
*HilbK3> let i = 1 :: Int
*HilbK3> let x = (PartLambda [2,2,1,1], [0,0,0,0]) :: AnBase
*HilbK3> let y = (PartLambda [2,1,1,1,1], [i,0,0,0,0]) :: AnBase
*HilbK3> cupInt x y
[(([2-1-1-1-1],[0,23,1,0,0]),-2),(([2-2-2],[1,0,0]),1),
(([3-2-1],[1,0,0]),2),(([4-1-1],[1,0,0]),1)]
\end{Verbatim}
\item We do the case $n=4,\ i=j=1$. 
\begin{Verbatim}[fontsize=\small]
*HilbK3> let i = 1 :: Int; let j = 1 :: Int
*HilbK3> let x = (PartLambda [2,1,1], [i,0,0]) :: AnBase
*HilbK3> let y = (PartLambda [1,1,1,1], [j,0,0,0]) :: AnBase
*HilbK3> cupInt x y
[(([2-1-1],[1,1,0]),1),(([3-1],[1,0]),1)]
\end{Verbatim}
\item We do the case $n=4$. 
\begin{Verbatim}[fontsize=\small]
*HilbK3> let d = (PartLambda [2,1,1], [0,0,0]) :: AnBase
*HilbK3> let y = (PartLambda [1,1,1,1], [23,0,0,0]) :: AnBase
*HilbK3> [ t | t <- cupInt d d, fst t == y]
[(([1-1-1-1],[23,0,0,0]),-3)]
\end{Verbatim}
\item We do the case $n=5$. 
\begin{Verbatim}[fontsize=\small]
*HilbK3> let x = (PartLambda [2,2,1], [0,0,0]) :: AnBase
*HilbK3> let y = (PartLambda [1,1,1,1,1], [23,23,0,0,0]) :: AnBase
*HilbK3> [ t | t <- cupInt x x, fst t == y]
[(([1-1-1-1-1],[23,23,0,0,0]),3)]
\end{Verbatim}
\end{enumerate}
\end{example}

\subsection{What the code does}
The goal is to multiply two elements in $H^*(S\hilb{n},\IZ)$. To do this, one has to execute the following steps:
\begin{enumerate}
 \item Compute the base change matrices $\psi_{\rho\nu}$ and $\psi_{\nu\rho}^{-1}$ between monomial and power sum symmetric functions.
 \item Provide a basis and the ring structure of $A=H^*(S,\IZ)$.
 \item Create a data structure for elements in $A\hilb{n}$ and $A\{S_n\}$.
 \item Implement the multiplication in $A\{S_n\}$, \ie the map $m_{\pi,\tau}$ from Definition \ref{model}.
 \item Implement the symmetrisation $A\hilb{n} = A\{S_n\}^{S_n}$.
 \item Use the isomorphism from Theorem \ref{LSThm} to get the ring structure of $A\hilb{n}$.
 \item Write an element in $H^*(S\hilb{n},\IZ)$ as a linear combination of products of creation operators acting on the vacuum, using Theorem \ref{QinWangTheorem}.
\end{enumerate}
We now describe, where to find these steps in the code.
\begin{enumerate}
 \item The $\psi_{\rho\nu}$ are computed by the function \verb|monomialPower| in the module \verb|SymmetricFunctions.hs|, using the theory from \cite[Sect.~3.7]{Lascoux}. The idea is to use the scalar product on the space of symmetric functions, so that the power sums become orthogonal: $(p_\lambda,p_\mu) = z_\lambda \delta_{\lambda\mu}$. The values for $(p_\lambda,m_\mu)$ are given by \cite[Lemma~3.7.1]{Lascoux}, so we know how to get the matrix $\psi_{\nu\rho}^{-1}$. Since it is triangular with respect to some ordering of partitions, matrix inversion is easy.
 \item The ring structure of $H^*(S,\IZ)$ is stored in the module \verb|K3.hs|. The only nontrivial multiplications are the products of two elements in $H^2(S,\IZ)$, where the intersection matrix is composed by the matrices for the hyperbolic and the $E_8$ lattice. The cup product and the adjoint comultiplication from Definition~\ref{comult} are implemented by the methods \verb|cup| and \verb|cupAd|.
 \item The data structures for basis elements of $A\hilb{n}$ and $A\{S_n\}$ are given by \verb|AnBase| and \verb|SnBase| in the module \verb|HilbK3.hs|. Linear combinations of basis elements are always stored as lists of pairs, each pair consisting of a basis element and a scalar factor.
 \item The function $m_{\pi,\tau}$ from Definition~\ref{model} is computed by the method \verb|multSn|. It contains the following substeps: First, the orbits of $\left<\pi,\tau\right>$ are computed recursively by glueing together the orbits of $\pi$ if they have both non-emtpty intersection with an orbit of $\tau$. Second, the composition $\pi\tau$ is computed using a method from the external library \verb|Data.Permute|. Third, the functions $f^{\pi,\left<\pi,\tau\right>}$ and $f_{\left<\pi,\tau\right>,\pi\tau}$ using the (co--)products from \verb|K3.hs|.
 \item The symmetrisation morphism is implemented by \verb|toSn|. We don't konw a better way to do this than the naive approach which is summation over all elements in $S_n$.
 \item The multiplication in $A\hilb{n}$ is carried out by the method \verb|multAn|.
 \item The base change matrices between the canonical base of $A\hilb{n}$ and the base of $H^*(S\hilb{n},\IZ)$ are given by \verb|creaInt| and \verb|intCrea|. By composing \verb|multAn| with these matrices, one gets the desired multiplication in $H^*(S\hilb{n},\IZ)$, called \verb|cupInt|.
\end{enumerate}

\subsection{Module for cup product structure of K3 surfaces} 
Here the hyperbolic and the $E_8$ lattice and the bilinear form on the cohomology of a K3 surface are defined. Furthermore, cup products and their adjoints are implemented.
\begin{lstlisting}
-- a module for the integer cohomology structure of a K3 surface
module K3 (
	K3Domain,
	degK3,
	rangeK3,
	oneK3, xK3,
	cupLSparse,
	cupAdLSparse
	) where

import Data.Array
import Data.List
import Data.MemoTrie

-- type for indexing the cohomology base
type K3Domain = Int

rangeK3 = [0..23] :: [K3Domain]

oneK3 = 0 :: K3Domain
xK3 = 23 :: K3Domain

rangeK3Deg :: Int -> [K3Domain]
rangeK3Deg 0 = [0]
rangeK3Deg 2 = [1..22]
rangeK3Deg 4 = [23]
rangeK3Deg _ = []

delta i j = if i==j then 1 else 0

-- degree of the element of H^*(S), indexed by i
degK3 :: (Num d) => K3Domain -> d
degK3 0 = 0 
degK3 23 = 4
degK3 i = if i>0 && i < 23 then 2 else error "Not a K3 index"

-- the negative e8 intersection matrix
e8 = array ((1,1),(8,8)) $
	zip [(i,j) | i <- [1..8],j <-[1..8]] [
	-2, 1, 0, 0, 0, 0, 0, 0,
	1, -2, 1, 0, 0, 0, 0, 0,
	0, 1, -2, 1, 0, 0, 0, 0,
	0, 0, 1, -2, 1, 0, 0, 0,
	0, 0, 0, 1, -2, 1, 1, 0,
	0, 0, 0, 0, 1, -2, 0, 1,
	0, 0, 0, 0, 1, 0, -2, 0,
	0, 0, 0, 0, 0, 1, 0, -2 :: Int]

-- the inverse matrix of e8
inve8 = array ((1,1),(8,8)) $
	zip [(i,j) | i <- [1..8],j <-[1..8]] [
	-2, -3, -4, -5, -6, -4, -3, -2, 
	-3, -6, -8,-10,-12, -8, -6, -4,
	-4, -8,-12,-15,-18,-12, -9, -6, 
	-5,-10,-15,-20,-24,-16,-12, -8,
	-6,-12,-18,-24,-30,-20,-15,-10,
	-4, -8,-12,-16,-20,-14,-10, -7, 
	-3, -6, -9,-12,-15,-10, -8, -5, 
	-2, -4, -6, -8,-10, -7, -5, -4 :: Int]

-- hyperbolic lattice
u 1 2 = 1
u 2 1 = 1
u 1 1 = 0
u 2 2 = 0
u i j = undefined

-- cup product pairing for K3 cohomology
bilK3 :: K3Domain -> K3Domain -> Int
bilK3 ii jj = let 
	(i,j) = (min ii jj, max ii jj) 
	in
	if (i < 0) || (j > 23) then undefined else
	if (i == 0) then delta j 23 else
	if (i >= 1) && (j <= 2) then u i j else
	if (i >= 3) && (j <= 4) then u (i-2) (j-2) else
	if (i >= 5) && (j <= 6) then u (i-4) (j-4) else
	if (i >= 7) && (j <= 14) then e8 ! ((i-6), (j-6)) else
	if (i >= 15) && (j<= 22) then e8 ! ((i-14), (j-14))  else
	0

-- inverse matrix to cup product pairing
bilK3inv :: K3Domain -> K3Domain -> Int
bilK3inv ii jj = let 
	(i,j) = (min ii jj, max ii jj) 
	in
	if (i < 0) || (j > 23) then undefined else
	if (i == 0) then delta j 23 else
	if (i >= 1) && (j <= 2) then u i j else
	if (i >= 3) && (j <= 4) then u (i-2) (j-2) else
	if (i >= 5) && (j <= 6) then u (i-4) (j-4) else
	if (i >= 7) && (j <= 14) then inve8 ! ((i-6), (j-6)) else
	if (i >= 15) && (j<= 22) then inve8 ! ((i-14), (j-14))  else
	0 

-- cup product with two factors
-- a_i * a_j = sum [cup k (i,j) * a_k | k<- rangeK3]
cup :: K3Domain -> (K3Domain,K3Domain) -> Int
cup = memo2 r where
	r k (0,i) = delta k i
	r k (i,0) = delta k i
	r _ (i,23) = 0
	r _ (23,i) = 0
	r 23 (i,j) =  bilK3 i j
	r _ _ = 0

-- indices where the cup product does not vanish
cupNonZeros :: [ (K3Domain,(K3Domain,K3Domain)) ]
cupNonZeros = [ (k,(i,j)) | i<-rangeK3, j<-rangeK3, k<-rangeK3, cup k (i,j) /= 0]

-- cup product of a list of factors
cupLSparse :: [K3Domain] -> [(K3Domain,Int)]
cupLSparse = cu . filter (/=oneK3) where
	cu [] = [(oneK3,1)]; cu [i] = [(i,1)]
	cu [i,j] = [(k,z) | k<-rangeK3, let z = cup k (i,j), z/=0]
	cu _ = []

-- comultiplication, adjoint to the cup product
-- Del a_k = sum [cupAd (i,j) k * a_i `tensor` a_k | i<-rangeK3, j<-rangeK3]
cupAd :: (K3Domain,K3Domain) -> K3Domain -> Int
cupAd = memo2 ad where 
	ad (i,j) k = negate $ sum [bilK3inv i ii * bilK3inv j jj 
		* cup kk (ii,jj) * bilK3 kk k |(kk,(ii,jj)) <- cupNonZeros ]

-- n-fold comultiplication
cupAdLSparse :: Int -> K3Domain -> [([K3Domain],Int)]
cupAdLSparse = memo2 cals where
	cals 0 k = if k == xK3 then [([],1)] else []
	cals 1 k = [([k], 1)]
	cals 2 k = [([i,j],ca) | i<-rangeK3, j<-rangeK3, let ca = cupAd (i,j) k, ca /=0]
	cals n k = clean [(i:r,v*w) |([i,j],w)<-cupAdLSparse 2 k, (r,v)<-cupAdLSparse(n-1) j]
	clean = map (\g -> (fst$head g, sum$(map snd g))). groupBy cg.sortBy cs  
	cs = (.fst).compare.fst; cg = (.fst).(==).fst

\end{lstlisting}

\subsection{Module for handling partitions} 
This module defines the data structures and elementary methods to handle partitions. We define both partitions written as descending sequences of integers ($\lambda$-notation) and as sequences of multiplicities ($\alpha$-notation).
\begin{lstlisting}
{-# LANGUAGE TypeOperators, TypeFamilies #-}

-- implements data structure and basic functions for partitions
module Partitions where

import Data.Permute
import Data.Maybe
import qualified Data.List 
import Data.MemoTrie

class (Eq a, HasTrie a) => Partition a where
	-- length of a partition
	partLength :: Integral i => a -> i 
	
	-- weight of a partition
	partWeight :: Integral i => a -> i
	
	-- degree of a partition = weight - length
	partDegree :: Integral i => a -> i
	partDegree p = partWeight p - partLength p
	
	-- the z, occuring in all papers
	partZ :: Integral i => a -> i
	partZ = partZ.partAsAlpha
	
	-- conjugated partition
	partConj :: a -> a
	partConj = res. partAsAlpha where
		make l (m:r) = l : make (l-m) r
		make _ [] = []
		res (PartAlpha r) = partFromLambda $ PartLambda $ make (sum r) r
	
	-- empty partition
	partEmpty :: a
	
	-- transformation to alpha-notation
	partAsAlpha :: a -> PartitionAlpha
	-- transformation from alpha-notation
	partFromAlpha :: PartitionAlpha -> a
	-- transformation to lambda-notation
	partAsLambda :: a -> PartitionLambda Int
	-- transformation from lambda-notation
	partFromLambda :: (Integral i, HasTrie i) => PartitionLambda i -> a
	
	-- all permutationens of a certain cycle type
	partAllPerms :: a -> [Permute]
	
-----------------------------------------------------------------------------------------

-- data type for partitiones in alpha-notation
-- (list of multiplicities)
newtype PartitionAlpha = PartAlpha { alphList::[Int] }

-- reimplementation of the zipWith function
zipAlpha op (PartAlpha a) (PartAlpha b) = PartAlpha $ z a b where
	z (x:a) (y:b) = op x y : z a b
	z [] (y:b) = op 0 y : z [] b
	z (x:a) [] = op x 0 : z a []
	z [] [] = []

-- reimplementation of the (:) operator
alphaPrepend 0 (PartAlpha []) = partEmpty
alphaPrepend i (PartAlpha  r) = PartAlpha (i:r)

-- all partitions of a given weight
partOfWeight :: Int -> [PartitionAlpha]
partOfWeight = let
	build n 1 acc = [alphaPrepend n acc]
	build n c acc = concat [ build (n-i*c) (c-1) (alphaPrepend i acc) | i<-[0..div n c]] 
	a 0 = [PartAlpha []]
	a w =  if w<0 then [] else  build w w partEmpty
	in memo a

-- all partitions of given weight and length
partOfWeightLength = let
	build 0 0 _ = [partEmpty]
	build w 0 _ = []
	build w l c = if l > w || c>w then [] else
		concat [ map (alphaPrepend i) $ build (w-i*c) (l-i) (c+1) 
			| i <- [0..min l $ div w c]]
	a w l = if w<0 || l<0 then [] else build w l 1
	in memo2 a

-- determines the cycle type of a permutation
cycleType :: Permute -> PartitionAlpha
cycleType p = let 
	lengths = Data.List.sort $ map Data.List.length $ cycles p
	count i 0 [] = partEmpty
	count i m [] = PartAlpha [m]
	count i m (x:r) = if x==i then count i (m+1) r 
		else alphaPrepend m (count (i+1) 0 (x:r)) 
	in count 1 0 lengths

-- constructs a permutation from a partition
partPermute :: Partition a => a -> Permute
partPermute = let
	make l n acc (PartAlpha x) = f x where
		f [] = cyclesPermute n acc 
		f (0:r) = make (l+1) n acc $ PartAlpha r
		f (i:r) = make l (n+l) ([n..n+l-1]:acc) $ PartAlpha ((i-1):r)
	in make 1 0 [] . partAsAlpha

instance Partition PartitionAlpha where
	partWeight (PartAlpha r) = fromIntegral $ sum $ zipWith (*) r [1..]
	partLength (PartAlpha r) = fromIntegral $ sum r
	partEmpty = PartAlpha []
	partZ (PartAlpha l) = foldr (*) 1 $ 
		zipWith (\a i-> factorial a*i^a) (map fromIntegral l) [1..] where
			factorial n = if n==0 then 1 else n*factorial(n-1)
	partAsAlpha = id
	partFromAlpha = id
	partAsLambda (PartAlpha l) = PartLambda $ reverse $ f 1 l where
		f i [] = []
		f i (0:r) = f (i+1) r
		f i (m:r) = i : f i ((m-1):r)
	partFromLambda = lambdaToAlpha
	partAllPerms = partAllPerms . partAsLambda

instance Eq PartitionAlpha where
	PartAlpha p == PartAlpha q = findEq p q where
		findEq [] [] = True
		findEq (a:p) (b:q) = (a==b) && findEq p q
		findEq [] q = isZero q
		findEq p [] = isZero p 
		isZero = all (==0) 

instance Ord PartitionAlpha where
	compare a1 a2 = compare (partAsLambda a1) (partAsLambda a2)

instance Show PartitionAlpha where 
	show p = let
		leftBracket = "(|"  
		rightBracket = "|)" 
		rest [] = rightBracket
		rest [i] = show i ++ rightBracket
		rest (i:q) = show i ++ "," ++ rest q
		in leftBracket ++ rest (alphList p) 

instance HasTrie PartitionAlpha where
	newtype PartitionAlpha :->: a =  TrieType { unTrieType :: [Int] :->: a }
	trie f = TrieType $ trie $ f . PartAlpha
	untrie f =  untrie (unTrieType f) . alphList
	enumerate f  = map (\(a,b) -> (PartAlpha a,b)) $ enumerate (unTrieType f)

-----------------------------------------------------------------------------------------

-- data type for partitions in lambda-notation
-- (descending list of positive numbers)
newtype PartitionLambda i = PartLambda { lamList :: [i] }

lambdaToAlpha :: Integral i => PartitionLambda i -> PartitionAlpha
lambdaToAlpha (PartLambda []) = PartAlpha[] 
lambdaToAlpha (PartLambda (s:p)) = lta 1 s p [] where
	lta _ 0 _ a = PartAlpha a
	lta m c [] a = lta 0 (c-1) [] (m:a)
	lta m c (s:p) a = if c==s then lta (m+1) c p a else 
		lta 0 (c-1) (s:p) (m:a)

instance (Integral i, HasTrie i) => Partition (PartitionLambda i) where
	partWeight (PartLambda r) = fromIntegral $ sum r
	partLength (PartLambda r) = fromIntegral $ length r
	partEmpty = PartLambda []
	partAsAlpha = lambdaToAlpha
	partAsLambda (PartLambda r) = PartLambda $ map fromIntegral r
	partFromAlpha (PartAlpha l) = PartLambda $ reverse $ f 1 l where
		f i [] = []
		f i (0:r) = f (i+1) r
		f i (m:r) = i : f i ((m-1):r)
	partFromLambda (PartLambda r) = PartLambda $ map fromIntegral r
	partAllPerms (PartLambda l) = it $ Just $ permute $ partWeight $ PartLambda l where
		it (Just p) = if Data.List.sort (map length $ cycles p) == r 
			then p : it (next p) else it (next p)
		it Nothing = []
		r = map fromIntegral $ reverse l

instance (Eq i, Num i) => Eq (PartitionLambda i) where
	PartLambda p == PartLambda q = findEq p q where
		findEq [] [] = True
		findEq (a:p) (b:q) = (a==b) && findEq p q
		findEq [] q = isZero q
		findEq p [] = isZero p 
		isZero = all (==0) 

instance (Ord i, Num i) => Ord (PartitionLambda i) where
	compare p1 p2 = if weighteq == EQ then compare l1 l2 else weighteq where
		(PartLambda l1, PartLambda l2) = (p1, p2)
		weighteq = compare (sum l1) (sum l2)

instance (Show i) => Show (PartitionLambda i) where
	show (PartLambda p) = "[" ++ s ++ "]" where
		s = concat $ Data.List.intersperse "-" $ map show p

instance HasTrie i => HasTrie (PartitionLambda i) where
	newtype (PartitionLambda i) :->: a =  TrieTypeL { unTrieTypeL :: [i] :->: a }
	trie f = TrieTypeL $ trie $ f . PartLambda
	untrie f =  untrie (unTrieTypeL f) . lamList
	enumerate f  = map (\(a,b) -> (PartLambda a,b)) $ enumerate (unTrieTypeL f)
\end{lstlisting}
\subsection{Module for coefficients on Symmetric Functions} 
This module provides nothing but the base change matrices $\psi_{\lambda\mu}$ and $\psi^{-1}_{\mu\lambda}$ from Definition \ref{SymFun}.
\begin{lstlisting}
-- A module implementing base change matrices for symmetric functions
module SymmetricFunctions(
	monomialPower,
	powerMonomial,
	factorial
	) where

import Data.List 
import Data.MemoTrie
import Data.Ratio
import Partitions

-- binomial coefficients
choose n k = ch1 n k where
	ch1 = memo2 ch
	ch 0 0 = 1
	ch n k = if n<0 || k<0 then 0 else if k> div n 2 + 1 then ch1 n (n-k) else
		ch1(n-1) k + ch1 (n-1) (k-1)

-- multinomial coefficients
multinomial 0 [] = 1
multinomial n [] = 0
multinomial n (k:r) = choose n k * multinomial (n-k) r

-- factorial function
factorial 0 = 1
factorial n = n*factorial(n-1)

-- http://www.mat.univie.ac.at/~slc/wpapers/s68vortrag/ALCoursSf2.pdf , p. 48
-- scalar product between monomial symmetric functions and power sums
monomialScalarPower moI poI = (s * partZ poI) `div` quo where
	mI = partAsAlpha moI
	s = sum[a* moebius b | (a,b)<-finerPart mI (partAsLambda poI)]
	quo = product[factorial i| let PartAlpha l =mI, i<-l] 
	nUnder 0 [] = [[]]
	nUnder n [] = [] 
	nUnder n (r:profile) = concat[map (i:) $ nUnder (n-i) profile | i<-[0..min n r]]
	finerPart (PartAlpha a) (PartLambda l) = nub [(a`div` sym sb,sb) 
		| (a,b)<-fp 1 a l, let sb = sort b] where
		sym = s 0 []
		s n acc [] = factorial n
		s n acc (a:o) = if a==acc then s (n+1) acc o else factorial n * s 1 a o
		fp i [] l = if all (==0) l then [(1,[[]|x<-l])] else []
		fp i (0:ar) l = fp (i+1) ar l
		fp i (m:ar) l = [(v*multinomial m p,addprof p op) 
			| p <- nUnder m (map (flip div i) l), 
			(v,op) <- fp (i+1) ar (zipWith (\j mm -> j-mm*i) l p)] where
				addprof = zipWith (\mm l -> replicate mm i ++ l)
	moebius l = product [(-1)^c * factorial c | m<-l, let c = length m - 1]

-- base change matrix from monomials to power sums
-- no integer coefficients
-- m_j = sum [ p_i * powerMonomial i j | i<-partitions]
powerMonomial :: (Partition a, Partition b) => a->b->Ratio Int
powerMonomial poI moI = monomialScalarPower moI poI % partZ poI

-- base change matrix from power sums to monomials
-- p_j = sum [m_i * monomialPower i j | i<-partitions]
monomialPower :: (Partition a, Partition b, Num i) => a->b->i 
monomialPower lambda mu = fromIntegral $ numerator $ 
	memoizedMonomialPower (partAsLambda lambda) (partAsLambda mu)  
memoizedMonomialPower = memo2 mmp1 where
	mmp1 l m  = if partWeight l == partWeight m then mmp2 (partWeight m) l m else 0 
	mmp2 w l m = invertLowerDiag (map partAsLambda $ partOfWeight w) powerMonomial l m

-- inversion of lower triangular matrix
invertLowerDiag vs a = ild where
	ild = memo2 inv
	delta i j = if i==j then 1 else 0
	inv i j | i<j = 0
		| otherwise = (delta i j - sum [a i k * ild k j | k<-vs, i>k , k>= j]) / a i i
\end{lstlisting} 
\subsection{Module implementing cup products for Hilbert schemes} This is our main module. We implement the algebraic model developped by Lehn and Sorger and the change of base due to Qin and Wang. The cup product on the Hilbert scheme is computed by the function \texttt{cupInt}.
\begin{lstlisting}
-- implements the cup product according to Lehn-Sorger and Qin-Wang
module HilbK3 where

import Data.Array
import Data.MemoTrie
import Data.Permute hiding (sort,sortBy)
import Data.List
import qualified Data.IntMap as IntMap
import qualified Data.Set as Set
import Data.Ratio
import K3
import Partitions
import SymmetricFunctions

-- elements in A^[n] are indexed by partitions, with attached elements of the base K3
-- is also used for indexing H^*(Hilb, Z)
type AnBase = (PartitionLambda Int, [K3Domain])

-- elements in A{S_n} are indexed by permutations, in cycle notation,
-- where to each cycle an element of the base K3 is attached, see L-S (2.5)
type SnBase = [([Int],K3Domain)]

-- an equivalent to partZ with painted partitions
-- counts multiplicites that occur, when the symmetrization operator is applied
anZ :: AnBase -> Int
anZ (PartLambda l, k) = comp 1 (0,undefined) 0 $ zip l k where
	comp acc old m (e@(x,_):r) | e==old = comp (acc*x) old (m+1) r
		| otherwise = comp (acc*x*factorial m) e 1 r
	comp acc _ m [] = factorial m * acc

-- injection of A^[n] in A{S_n}, see L-S 2.8
-- returns a symmetrized vector of A{S_n}
toSn :: AnBase -> ([SnBase],Int)
toSn = makeSn where
	allPerms = memo p where 
		p n = map (array (0,n-1). zip [0..]) (permutations [0..n-1]) 
	shape l = (map (forth IntMap.!) l, IntMap.fromList $ zip [1..] sl) where
		sl = map head$ group $ sort l; 
		forth = IntMap.fromList$ zip sl [1..]
	symmetrize :: AnBase -> ([[([Int],K3Domain)]],Int)
	symmetrize (part,l) = (perms, toInt $ factorial n % length perms)  where 
		perms = nub [sortSn$ zipWith (\c cb ->(ordCycle $ map(p!)c, cb) ) cyc l 
			| p <- allPerms n]
		cyc = sortBy ((.length).flip compare.length) $ cycles $ partPermute part
		n = partWeight part
	ordCycle cyc = take l $ drop p $ cycle cyc where
		(m,p,l) = foldl findMax (-1,-1,0) cyc
		findMax (m,p,l) ce = if m<ce then (ce,l,l+1) else (m,p,l+1)
	sortSn = sortBy	compareSn  where
		compareSn (cyc1,class1) (cyc2,class2) = let
			cL = compare l2 $ length cyc1 ; l2 = length cyc2
			cC = compare class2 class1
			in if cL /= EQ then cL else 
				if cC /= EQ then cC else compare cyc2 cyc1  
	mSym = memo symmetrize
	makeSn (part,l) = ([ [(z,im IntMap.! k) | (z,k) <- op ]|op <- res],m)  where
		(repl,im) = shape l
		(res,m) = mSym (part,repl)

-- multiplication in A{S_n}k, see L-S, Prop 2.13
multSn :: SnBase -> SnBase -> [(SnBase,Int)]
multSn l1 l2 = tensor $ map m cmno where
	-- determines the orbits of the group generated by pi, tau
	commonOrbits :: Permute -> Permute -> [[Int]]
	commonOrbits pi tau = Data.List.sortBy ((.length).compare.length) orl where
		orl = foldr (uni [][]) (cycles pi) (cycles tau) 
		uni i ni c []  = i:ni
		uni i ni c (k:o) = if Data.List.intersect c k == [] 
			then uni i (k:ni) c o else uni (i++k) ni c o
	pi1 = cyclesPermute n $ cy1 ; cy1 = map fst l1; n = sum $ map length cy1
	pi2 = cyclesPermute n $ map fst l2
	set1 = map (\(a,b)->(Set.fromList a,b)) l1; 
	set2 = map (\(a,b)->(Set.fromList a,b)) l2
	compose s t = swapsPermute (max (size s) (size t)) (swaps s ++ swaps t)
	tau = compose pi1 pi2
	cyt = cycles tau ; 
	cmno = map Set.fromList $ commonOrbits pi1 pi2; 
	m or = fdown where
		sset12 = [xv | xv <-set1++set2, Set.isSubsetOf (fst xv) or]
		-- fup and fdown correspond to the images of the maps described in L-S (2.8)
		fup = cupLSparse $ map snd sset12 ++ replicate def xK3
		t = [c | c<-cyt, Set.isSubsetOf (Set.fromList c) or]
		fdown = [(zip t l,v*w*24^def)| (r,v) <- fup, (l,w)<-cupAdLSparse(length t) r] 
		def = toInt ((Set.size or + 2 - length sset12 - length t)%2)

-- tensor product for a list of arguments
tensor :: Num a =>  [[([b],a)]] -> [([b],a)]
tensor [] = [([],1)]
tensor (t:r) = [(y++x,w*v) |(x,v)<-tensor r, (y,w) <- t ]

-- multiplication in A^[n]
multAn :: AnBase -> AnBase -> [(AnBase,Int)]
multAn a = multb where
	(asl,m) = toSn a
	toAn sn =(PartLambda l, k) where 
		(l,k)= unzip$ sortBy (flip compare)$ map (\(c,k)->(length c,k)) sn
	multb (pb,lb) = map ungroup$ groupBy ((.fst).(==).fst) $sort elems where
		ungroup g@((an,_):_) = (an, m*(sum $ map snd g) )
		bs = zip (sortBy ((.length).flip compare.length) $cycles $ partPermute pb) lb
		elems = [(toAn cs,v) | as <- asl, (cs,v) <- multSn as bs]

-- integer base to ordinary base, see Q-W, Thm 1.1
intCrea :: AnBase -> [(AnBase,Ratio Int)]
intCrea = map makeAn. tensor. construct where
	memopM = memo pM
	pM pa = [(pl,v)| p@(PartLambda pl)<-map partAsLambda$ partOfWeight (partWeight pa), 
		let v = powerMonomial p pa, v/=0]
	construct pl = onePart pl : xPart pl : 
		[ [(zip l $ repeat a,v)| (l,v)<- memopM (subpart pl a)] |a<-[1..22]] 
	onePart pl = [(zip l$ repeat oneK3, 1%partZ p)] where 
		p@(PartLambda l) = subpart pl oneK3
	xPart pl = [(zip l$ repeat xK3, 1)] where 
		(PartLambda l) = subpart pl xK3
	makeAn (list,v) = ((PartLambda x,y),v) where 
		(x,y) = unzip$ sortBy (flip compare) list 

-- ordinary base to integer base, see Q-W, Thm 1.1
creaInt :: AnBase -> [(AnBase, Int)]
creaInt = map makeAn. tensor. construct where
	memomP = memo mP
	mP pa = [(pl,v)| p@(PartLambda pl)<-map partAsLambda$ partOfWeight (partWeight pa), 
		let v = monomialPower p pa, v/=0]
	construct pl = onePart pl : xPart pl : 
		[ [(zip l $ repeat a,v)| (l,v)<- memomP (subpart pl a)] |a<-[1..22]] 
	onePart pl = [(zip l$ repeat oneK3, partZ p)] where 
		p@(PartLambda l) = subpart pl oneK3
	xPart pl = [(zip l$ repeat xK3, 1)] where 
		(PartLambda l) = subpart pl xK3
	makeAn (list,v) = ((PartLambda x,y),v) where 
		(x,y) = unzip$ sortBy (flip compare) list 

-- cup product for integral classes
cupInt :: AnBase -> AnBase -> [(AnBase,Int)]
cupInt a b = [(s,toInt z)| (s,z) <- y] where
	ia = intCrea a; ib = intCrea b
	x = sparseNub [(e,v*w*fromIntegral z) | (p,v) <- ia, 
		let m = multAn p, (q,w) <- ib, (e,z)<- m q] 
	y = sparseNub [(s,v*fromIntegral w) | (e,v) <- x, (s,w) <- creaInt e]

-- helper function, adds duplicates in a sparse vector
sparseNub :: (Num a) => [(AnBase, a)] -> [(AnBase,a)] 
sparseNub = map (\g->(fst$head g, sum $map snd g)).groupBy ((.fst).(==).fst). 
	sortBy ((.fst).compare.fst)

-- cup product for integral classes from a list of factors
cupIntList :: [AnBase] -> [(AnBase,Int)]
cupIntList = makeInt. ci . cL where
	cL [b] = intCrea b
	cL (b:r) = x where
		ib = intCrea b
		x = sparseNub [(e,v*w*fromIntegral z) | 
			(p,v) <- cL r, let m = multAn p, (q,w) <- ib, (e,z)<-m q]
	makeInt l = [(e,toInt z) | (e,z) <- l]
	ci l = sparseNub [(s,v*fromIntegral w) | (e,v) <- l, (s,w) <- creaInt e]

-- degree of a base element of cohomology
degHilbK3 :: AnBase -> Int
degHilbK3 (lam,a) = 2*partDegree lam + sum [degK3 i | i<- a]

-- base elements in Hilb^n(K3) of degree d 
hilbBase :: Int -> Int -> [AnBase]
hilbBase = memo2 hb where
	hb n d = sort $map ((\(a,b)->(PartLambda a,b)).unzip) $ hilbOperators n d  

-- all possible combinations of creation operators of weight n and degree d
hilbOperators :: Int -> Int -> [[ (Int,K3Domain) ]]
hilbOperators = memo2 hb where 
	hb 0 0 = [[]] -- empty product of operators
	hb n d = if n<0 || odd d || d<0 then [] else 
		nub $ map (Data.List.sortBy (flip compare)) $ f n d
	f n d = [(nn,oneK3):x | nn <-[1..n], x<-hilbOperators(n-nn)(d-2*nn+2)] ++
		[(nn,a):x | nn<-[1..n], a <-[1..22], x<-hilbOperators(n-nn)(d-2*nn)] ++
		[(nn,xK3):x | nn <-[1..n], x<-hilbOperators(n-nn)(d-2*nn-2)] 

-- helper function
subpart :: AnBase -> K3Domain -> PartitionLambda Int
subpart (PartLambda pl,l) a = PartLambda $ sb pl l where
	sb [] _ = []
	sb pl [] = sb pl [0,0..]
	sb (e:pl) (la:l) = if la == a then e: sb pl l else sb pl l

-- converts from Rational to Int
toInt :: Ratio Int -> Int
toInt q = if n ==1 then z else error "not integral" where 
	(z,n) =(numerator q, denominator q)
\end{lstlisting}

\bibliographystyle{amsplain}

\end{document}